\documentclass[a4paper,reqno,12pt]{amsart}
	\usepackage[latin1]{inputenc}
    \usepackage[T1]{fontenc}
    \usepackage[english]{babel}
    \usepackage{appendix}
    \usepackage[english]{minitoc}
    \usepackage[nice]{nicefrac}
		\usepackage[top=3cm, bottom=3cm, left=3cm,right=3cm]{geometry}
\numberwithin{equation}{section}
\makeatletter\@addtoreset{equation}{section} 
	\usepackage{amsmath}
    \usepackage{amssymb}
    \usepackage{amsbsy} 
    \usepackage{latexsym, amsfonts}
    \usepackage{graphics}
    \usepackage{ulem}
    \usepackage{hhline}
    \usepackage{dsfont}
    \usepackage{mathrsfs}
    \usepackage{color}
    \usepackage{fancyhdr}
    \usepackage{rotating}
    \usepackage{fancybox}
    \usepackage{colortbl}
    \usepackage{pifont}
    \usepackage{setspace}
    \usepackage{enumerate}
    \usepackage{multicol}
    \usepackage{varioref}
    \usepackage{lmodern}
    \usepackage{textcomp}
    \usepackage{euscript}
    \usepackage[pdftex]{hyperref}
			\newtheorem{theorem}{Theorem}[section]

			\newtheorem{lemma}[theorem]{Lemma}
			\newtheorem{proposition}[theorem]{Proposition}
			
			\newtheorem{corollary}[theorem]{Corollary}
			\newtheorem{remark}[theorem]{Remark}

\newcommand{\C}{\mathbb C}       \newcommand{\Z}{\mathbb Z} 	 
\newcommand{\Cch}{\mathcal C}
\newcommand{\Di}{\mathbb{D}}

\newcommand{\gm}{\gamma}

\newcommand{\AnWDi}{H^{2,\gamma}}
\newcommand{\AnTWDi}{\mathcal{A}^{2,\gamma}}

\newcommand{\norm}[1]{{\left\|{#1}\right\|}}    
   \newcommand{\scal}[1]{{\left\langle{#1}\right\rangle}}
    \newcommand{\bz}{\overline{z}}
\newcommand{\bxi}{\overline{\xi}}  
   
\newcommand{\minmn}{m\wedge n}  
 
\begin{document}

\title[]{Solid Cauchy transform on weighted poly-Bergman spaces}
\author[R. El Harti, A. Elkachkouri, A. Ghanmi]{R. El Harti,, A. ElKachkouri, A. Ghanmi}
 
 \email{rachid.elharti@uhp.ac.ma}
 \email{elkachkouri.abdelatif@gmail.com} 
 \email{allal.ghanmi@um5.ac.ma} 
 \address{Analysis, P.D.E. and Spctral Geometry, Lab. M.I.A.-S.I., CeReMAR, Department of Mathematics, P.O. Box 1014, Faculty of Sciences, Mohammed V University in Rabat, Morocco}
 \address{Department of Mathematics and Computer Sciences, Faculty of Sciences and
 	Techniques,
 	University Hassan I, BP 577 Settat, Morocco}

\begin{abstract}
	The so-called weighted solid Cauchy transform, from inside the unit disc into the complement of its closure, is considered and their basic properties such as boundedness is studied for appropriate probability measures. 
	The action the disc polynomials is explicitly computed and used to describe the range of its restriction on the weighted true poly-Bergman spaces. 
\end{abstract}


 

\maketitle
\section{Introduction}
 
 The weighted Cauchy transform 
\begin{align}\label{CT}
\mathcal{C}_{s}^{\mu} f (z)  :=  \frac 1\pi \int_{\Omega} \frac{f(\xi)}{\bz-\bxi} d\mu(z); \,  z\notin \overline{\Omega},
\end{align}
for $z$ in the complement of the closure of the unit disc $\Omega$, $\overline{\Omega}^c:=\C\setminus \overline{\Omega}$,
associated to given measure $\mu$ on a bounded $\Omega$ in the complex plane, can be seen as a specific adjoint of the classical weighted Cauchy transform on $b\Omega$ \cite[p. 89]{Bell2016}.
The importance  of these operator lies  in the fact that its kernel function is the fundamental solution of the $\overline{\partial}$ operator and is  closely connected to the Green function for Dirichlet Laplacian for specific $\Omega$.

The description of the range of  $\mathcal{C}_{s}^{\mu}$, corresponding to specific $\mu$, when acting on the  different  standard spaces of analytic functions was investigated by
many authors, see for instance \cite{Calderon1977,Lutsenko Yulmukhametov1991,NapalkovYulmukhametov1994,NapalkovYulmukhametov1997}.
For the Bergman space $A^2(\Omega)$  of analytic functions on $\Omega$, this problem has been solved by many authors, notably by Napalkov and Yulmukhametov \cite{NapalkovYulmukhametov1994} for Jordan domains and by Merenkov \cite{Merenkov2000,Merenkov2013}  for a large class of domains including those bounded by a Jordan curve $\partial \Omega$ with $area(\partial \Omega)=0$ or the integrable Jordan domains.
Obviously, the restriction of the solid Cauchy transform to $A^2(\Omega)$ is injective
and its image is contained in the space of analytic functions
on $\overline{\Omega}^c$ with $f(\infty)=0$. Moreover, it is a continuous operator \cite{Merenkov2000} from 
$A^2(\Omega )$ into the special Bergman-Sobolev space $B_1^2(\overline{\Omega}^c)$ defined as the space of all holomorphic functions in $\overline{\Omega}^c$ vanishing at infinity and whose derivative belongs to $A^2(\overline{\Omega}^c)$, 
$$B_1^2(\overline{\Omega}^c)=\{f; f^{'}\in A(\overline{\Omega}^c); f(\infty)=0\}.$$ 
More precisely, for domains with sufficiently smooth boundary, it is proved in 
 \cite{NapalkovYulmukhametov1994} that the solid Cauchy transform $\mathcal{C}_{s}:=\mathcal{C}_{s}^{1}$; $\mu =dxdy$ being the area measure, 
   maps the dual space $A^{2*}(\Omega)$  of $A^2(\Omega)$ continuously onto $B_1^2(\overline{\Omega}^c)$.
In \cite{Merenkov2000}, it is proved that $\mathcal{C}_{s}(A^{*}(\Omega))=B_1^2(\overline{\Omega}^c)$ remains valid if $\Omega$ is a quasidisc, i.e., an interior of a Jordan curve $\mathcal{J}$ for which there exists $M>0$ satisfying $diam(\mathcal{J}(a,b))\leq M|a-b|$ for any $a,b\in \mathcal{J}$. Here $\mathcal{J}(a,b)$ is an arc of the smaller diameters of $\mathcal{J}$ joining $a$ and $b$.
For bounded Jordan domain, it turns out to be a sufficient and necessary condition \cite{NapalkovYulmukhametov1997}. Namely,   
the Cauchy transformation is a surjective continuous operator from $A^{2*}(\Omega)$ onto $B_1^2(\overline{\Omega}^c)$  if and only if the domain $\Omega$  is a quasidisk.
 This characterization has been next used in \cite{IsaevYoulmukhametov2004}  to investigate the action of the Laplace transform on Bergman spaces.

 In the present paper, we consider similar problem for the  weighted solid Cauchy transform with respect to $\omega_\alpha(|\xi|^2): (1-|\xi|^2)^\alpha$, \cite{Bell2016} 
 \begin{align}\label{CT}
 \mathcal{C}_{s}^{\omega_\alpha} f (z)  :=  \frac 1\pi \int_{\Di} \frac{f(\xi)}{z-\xi} \omega_\alpha(|\xi|^2) dxdy; \quad z\in \overline{\Di}^c,
 \end{align}
  when acting on weighted true poly-Bergman spaces $\AnTWDi_n(\Di)$ in the unit disc $\Di$, which are specific generalization of the classical weighted Bergman space to the polyanalytic setting. Thus, our main task is to determinate the range of $\mathcal{C}_{s}$ when acting on the $n$-th weighted true poly-Bergman space $\AnTWDi_n(\Di)$ defined in \cite{ElHGhIn} 
 \begin{align}\label{SpaceA}\AnTWDi_{n}(\Di)= \AnWDi_n(\Di) \ominus \AnWDi_{n-1}(\Di)
 \end{align}
 for $n\geq 1$ and  $\mathcal{A}_{0}^{2,\gamma}(\Di) = \AnTWDi(\Di)$.
 Here $\AnWDi_m(\Di)$ denotes the space of polyanalytic functions of order $n+1$ that are square integrable on $\Di$ with respect to given radial weight function $\omega_\gamma(|\zeta|^2): (1-|\zeta|^2)^\gamma$, $\AnWDi_m(\Di) :=  
 ker \partial^{n+1}_{\bz} \cap L^{2,\gamma}(\Di)$, with $\partial_{\bz}$ is the standard Wirtinger derivative. 
The main result shows that the restriction of  ${\Cch_s}^{\omega_\alpha}$ to $\AnTWDi_{n}(\Di)$ is a finite dimensional vector space contained in the one spanned by 
$ z^{n-m+1}$; $m=0, 1, 2, \cdots , n .$ Its precise dimension depends on the quantization of $\gamma -\alpha$ and the order of polyanalyticity $n+1$. More precisely, it can be stated as follows

  \begin{theorem}\label{thm1}
  	Let $\gamma >-1$ and $\alpha >(\gamma-1)/0$. Then,   ${\Cch_s}^{\omega_\alpha}(\AnTWDi_{n}(\Di))$ to $\AnTWDi_{n}(\Di)$ is a finite dimensional vector of dimension $N=\dim({\Cch_s}^{\omega_\alpha}(\AnTWDi_{n}(\Di)))=\min (n,\alpha-\gamma+1)+1$ when $\gamma -\alpha \in \Z^-_0$ and $N=n+1$. otherwise.  
  \end{theorem}

  \begin{corollary} \label{cor1}
	Under the condition above, the null space of the restriction of $ {\Cch_s}^{\omega_\alpha}$ to $\AnTWDi_{n}(\Di)$ is spanned by the disc polynomials $\mathcal{R}^{\gamma}_{m,n}$; $m \geq  \min (n,\alpha-\gamma+1)$. 
\end{corollary}

 \begin{corollary} \label{cor2}
	The spaces  ${\Cch_s}^{\omega_\alpha}(\AnTWDi_{n}(\Di))_{n}$ form an increasing sequence of spaces in $L^{2,\gamma}(\Di)$.
\end{corollary}

%
%


This note is outlined as follows. In Section 2, we briefly review from \cite{Ramazanov2002,ElHartiElkGh2020A} some needed facts on the underlying true weighted poly-Bergman space $\AnTWDi_n(\Di)$. Section 3 is devoted to discuss the boundedness of the weighted solid Cauchy transform $\mathcal{C}_{s}^{\mu}$ for specific weight functions. In Section 4, we collect and establish needed tools on the action of $\mathcal{C}_{s}^{\omega_\alpha}$ on the weighted poly-Bergman spaces $\AnTWDi_n(\Di)$. The proof of Theorem \ref{thm1} and their corollaries is presented in Section 5.

\section{Preliminaries on weighted true poly-Bergman spaces}

Let  $\gamma > -1$ and denote by 
 $L^{2,\gamma}(\Di):=L^{2}\left(\Di,d\mu_{\gamma}\right)$ the Hilbert space
of complex-valued functions in the unit disk $\Di=\{z\in \C; \, |z|<1\}$ with the norm induced from the scalar product
$$ \scal{ f,g}_\gm := \int_{\Di} f(z) \overline{g(z)} (1-|z|^2)^{\gamma} dxdy ; \quad z = x+iy \in \Di. $$
Orthogonal decompositions of $L^{2,\gamma}(\Di)$ can be provided by means of the so-called polyanalytic functions \cite{Balk1991} which are solutions of the generalized Cauchy--Riemann equation on the unit disc $\Di$, $$\partial^{n+1}_{\bz} f = \frac{\partial^{n+1} f}{\partial \bz^{n+1}}  =\frac{1}{2} \left( \frac{\partial}{\partial x} + i \frac{\partial}{\partial y}\right)^{n+1}f =0.$$ 
In fact, we have the orthogonal Hilbertian decompositions 
$$\displaystyle L^{2,\gamma}(\Di) =  \bigoplus_{n=0}^\infty \AnTWDi_{n}(\Di) \quad  \mbox{and} \quad \AnWDi_n(\Di) = \bigoplus_{k=0}^n \mathcal{A}_{k}^{2,\gamma}(\Di),$$
where the space $\AnTWDi_{n}(\Di)$ are as in \eqref{SpaceA}. They are closed subspaces of $L^{2,\gamma}(\Di)$and form an orthogonal sequence of reproducing kernel Hilbert spaces.  
 An orthogonal basis of $\AnTWDi_{n}(\Di)$ is shown in \cite{ElHartiElkGh2020A}  to be given by the disc polynomials \cite{
	Koornwinder75,Koornwinder78,
	Dunkl84,
	Aharmim2015,ElHGhIn}  
\begin{align} \label{expZ}
\mathcal{R}_{m,n}^{\gamma}(z,\bar{z})
= m!n!\sum_{j=0}^{\minmn }
\frac{(-1)^{j} (1 -  z\bz)^{j}} {j! (\gamma+1)_j } \frac{ z^{m-j}}{(m-j)! } \frac{\bz^{n-j} }{(n-j)!} 
	\end{align}
	for varying $m=0,1,2, \cdots$.
Above, $m\wedge n=\min(m,n)$ and $(a)_k=a(a+1) \cdots (a+k-1)$ denotes the Pochhammer symbol.
The completeness of the system $\mathcal{R}_{m,n}^{\gamma}$ in $L^{2,\gamma}(\Di)$ was crucial in exploring the so-called weighted true poly-Bergman spaces. 

It should be mentioned here that for $\gamma=0$, the spaces $\mathcal{A}_{m}^{2,0}(\Di)$ reduces further to the poly-Bergman spaces studied in \cite{Ramazanov1999,Vasilevski2000,RozenblumVasilevski2019}. 

\section{Boundedness of ${\Cch_s}$}

Let $\omega$ be a weight function on the segment $(0,1)$ such that the associated moment 
\begin{align}\label{assumomega}
\gamma_{n}^{\omega} &:= \int_0^1 t^n \omega(t) dt \leq \gamma_{0}^{\omega} ; \quad n=0,1,2, \cdots , 
\end{align}
are finite. Associated to $\omega$ that we extend to a measure on the unit disc $\Di$ in the usual way by considering $\omega(|z|^2)$ for $z\in \Di$, we consider the corresponding weighted solid Cauchy transform $\left[   {\Cch_s}^{\omega}  (f) \right](z)$ in \eqref{CT}
with $\Omega =\Di$ and $z\in \overline{\Di}^c $. We consider its action on the Hilbert space  
$L^{2,A}(\Di):=L^2(\Di;A(|\xi|^2) d\lambda)$ of all complex-valued functions on $\Di$ such that are square integrable with respect to given measure $A(|\zeta|^2)d\lambda$, $d\lambda$ being the Lebesgue measure. In a similar way, we define the Hilbert space 
$L^{2,B}(\overline{\Di}^c):=L^2(\overline{\Di}^c;B(|\xi|^2) d\lambda)$ for given  weight function $B(|\xi|^2)$ in $\overline{\Di}^{c}$, arising from a weight function on $(1,+\infty)$. We denote by $\scal{ \cdot, \cdot }_{A}$ and $\scal{ \cdot, \cdot }_{B}$ the associated scalar product and by $\norm{ \cdot}_{A}$ and $\norm{ \cdot}_{B}$ the associated  associated norms, respectively.

In the sequel, we provide sufficient conditions on $\omega$, $A$ and $B$ to ${\Cch_s}$ be a bounded operator from $L^2_A(\Di)$ into $L^{2,B}(\overline{\Di}^c)$. Namely, we assume that 
\begin{align}\label{assumstar1}
V_{\omega^2/A}:= \int_0^1 \frac{\omega^{2}(t)}{A(t)}  dt < \infty 
\end{align}
and 
\begin{align}\label{assumstar2}
 W_{B} :=\int_0^1   \frac{B(1/t^2)}{t(1-t)^2} dt <\infty.
\end{align}

\begin{proposition} 
	Under \eqref{assumstar1}, the transform ${\Cch_s}^{\omega}$ is a well defined bounded operator from $ L^{2,A}(\Di) $ to  Hilbert space $L^{2,B}(\overline{\Di}^c)$.
\end{proposition} 

\begin{proof}
	Using Cauchy Schwartz inequality, it is not hard to show that ${\Cch_s}^{\omega} $ is well defined on $L^{2,A}(\Di)$. Indeed, for any $f\in L^{2,A}(\Di)$ we have
	\begin{align}
	\left| {\Cch_s}^{\omega}(f) (z)\right| &\leq  \dfrac{1}{\pi}  
	\left( \int_{\Di} \frac{\omega^{2} (|w|^2)}{|w-z|^{2}A(|w|^2)}d\lambda(w)\right)^{1/2} \norm{f}_A \nonumber
	\leq \dfrac{ \left( V_{\omega^2/A} \right)^{1/2}}{\pi(|z|-1)}   \norm{f}_A 
	\label{estimate1}
	\end{align} 
	which is finite by means of \eqref{assumstar1} and the fact that $|w-z|^{-2}$, for varying $w\in \Di$, satisfies  
	$|w-z|^{-2}\leq  (|z|-1)^{-2}$. 
	Therefore, it follows 
	\begin{align*}
	\norm{{\Cch_s}^{\omega} (f) }^{2}_{B} 
	&\leq    \frac{2}{\pi} V_{\omega^{2}/A}  
	\left( \int_1^{\infty} \dfrac{B(r^2)}{ (r-1)^2}   r dr \right)   \norm{f}_A^2  
	 \\& \leq \frac{2}{\pi} V_{\omega^{2}/A}  
	\left( \int_0^1 \frac{B(1/t^2)}{t(1-t)^{2}}
	dt \right)   \norm{f}_A^2 . 
	\end{align*}
	Then, one concludes for the boundedness of the solid Cauchy transform ${\Cch_s}^{\omega}$ from $L^{2,A}(\Di)$ into $L^{2,B}(\overline{\Di}^c)$ thanks to assumptions \eqref{assumstar1} and \eqref{assumstar2}. 
\end{proof}

\begin{remark} 
	A precise estimate for the norm operator of the solid weighted Cauchy transform ${\Cch_s}^{\omega}$ can be given for explicit wight functions satisfying assumptions \eqref{assumstar1} and \eqref{assumstar2}. For example for $A(t)=\omega(t)=\omega_\gamma(t) =(1-t)^\gamma$ and $B(t)= B_{a,b}(t):= t^a(t-1)^b$ with $-a<b<-1<\gamma$, the evaluation of the integrals giving $V_{\omega^{2}/A}$ and $W_{B}$  yields  
	\begin{equation}\label{estimate2}
	\norm{{\Cch_s}^{\omega}}^{2} \leq \frac{2^{1-b} \Gamma(2a+2b)\Gamma(-b-1)}{\pi (\gamma+1) \Gamma(2a+b-1)}    .
	\end{equation}
\end{remark}

\begin{remark} \label{remCond}
For $A(t)=\omega_\gamma(t)$ and $\omega(t)=\omega_\alpha(t)$, the corresponding quantity  $V_{\omega_\alpha^{2}/\omega_\gamma}$ is finite if and only if $ \gamma < 1 +2\alpha$. 
\end{remark}

\section{The range and null space of ${\Cch_s}^{\omega_\alpha}$}
In order to explore the basic properties of ${\Cch_s}^{\omega_\alpha}$  such as the description of its null space and the range of its restriction to the true weighted polyBergman spaces, we specify the weight functions $A(t)=\omega_\gm(t)=(1-t)^{\gm}$ and $B$, and we provide the explicit action of ${\Cch_s}^{\omega_\alpha}$ on the disc polynomials $\mathcal{R}_{m,n}^{\gamma}$. The advantage of considering such class of functions is that they form an an orthogonal basis of the Hilbert space  $L^{2,\gamma}(\Di)$. To this end, we begin by giving its action on the generic functions $$ e^{\ell}_{jk}(\xi,\bxi)=\xi^{j}\overline{\xi}^{k} (1-|\xi|^{2})^{\ell}$$
for nonnegative integers $j,k, \ell$. 

\begin{lemma}\label{lemaction} 
	We have 
	 \begin{align}\label{varep}
	\left\{ \begin{array}{ll}
	\left[{\Cch_s}^{\omega_\alpha} (e^{\ell}_{k+m,k}) \right](z)  =0  &  \quad \mbox{if } \quad m>0 \\ \quad \\
	\left[{\Cch_s}^{\omega_\alpha} (e^{\ell}_{k,k+m}) \right](z)  =  \gamma^{\omega}_{k,s} \frac{1}{z^{m+1}}                               &  \quad \mbox{if} \quad  m\geq 0 .
	\end{array}
	\right.
	\end{align}
	where $\gamma^{\omega}_{k,s}$ stands for 
	\begin{align}\label{momenw}
	\gamma^{\omega}_{k,s}=\int_{0}^{1}t^{k}(1-t)^{s}  \omega(t)dt.
	\end{align}
\end{lemma}

\begin{proof}
	Notice first that for any $z\in \overline{\Di}^c $  and $\xi \in D,$ we have $\xi/z\in \Di$. 
	Then, by expanding the kernel function, 
	we obtain
	\begin{align}
	\left[{\Cch_s}^{\omega_\alpha} (e^{s}_{jk}) \right](z)
	&=\dfrac{1}{\pi}\sum_{l=0}^{+\infty}\frac{1}{z^{l+1}}\int_{\Di}\xi^{j+l}(\overline{\xi})^{k}(1-|\xi|^{2})^{s}  \omega\left(|\xi|^2\right)d\lambda(\xi) \nonumber \\
	&=\sum_{l=0}^{+\infty}\frac{1}{z^{l+1}}\left( \int_{0}^{1}t^{k}(1-t)^{s}  \omega (t )dt\right)  \delta_{j+l,k} \nonumber 
	\\&=\varepsilon_{k-j}\frac{\gamma^{\omega}_{k,s}}{z^{k-j+1}}, \label{actionejks}
	\end{align}
	where 
	 \begin{align}\label{varep}
	\varepsilon_{p} =
	\left\{ \begin{array}{ll}
	1 &  \quad \mbox{if } \quad p\geq 0 \\
	0                               &  \quad \mbox{if} \quad  p <0 .
	\end{array}
	\right.
	\end{align}
	The last equalities follow by the use of polar coordinates $\xi=re^{i\theta}$, the fact that $\int_{0}^{2\pi}e^{i(m-n)\theta}d\theta=2\pi\delta_{n,m}   $ and the change of variable $r^{2}=t$, keeping in mind the definition of $\gamma^{\omega}_{k,s}$ and $
	\varepsilon_{p}$ given through \eqref{momenw} and \eqref{varep}, respectively. 
\end{proof}

Accordingly, it is clear from Lemma \ref{lemaction} that the holomorphic functions  ${\Cch_s}^{\omega_\alpha}(e^{s}_{jk})$   belong to $Ker({\Cch_s}^{\omega_\alpha})$, for any $s$ and any $j>k$.
Moreover, we assert the following 

\begin{proposition}\label{prop5assertions} 
	The function ${\Cch_s}^{\omega_\alpha}(e^{s}_{jk})$ belongs to $L^{2,B}(\overline{\Di}^c)$ and its square norm is given by 
	$$\norm{{\Cch_s}^{\omega_\alpha}(e^{s}_{jk}) }^{2}_{B}=\pi \varepsilon_{k-j}(\gamma^{\omega}_{k,s})^{2}  \int_0^1 \frac{B(1/u)}{u^{k-j+1}}  du.
	$$ Moreover, the system $\left( {\Cch_s}^{\omega_\alpha}(e^{s}_{jk})\right)_{j,s}$ is  orthogonal in $L^{2,B}(\overline{\Di}^c)$ for every fixed $k$. 
\end{proposition}

\begin{proof} For the proof, we need only to compute 
	$\scal{{\Cch_s}^{\omega_\alpha}(e^{s}_{jk}), {\Cch_s}^{\omega_\alpha}(e^{r}_{mn}) ^{\omega_\alpha}}_{B}$ for arbitrary $m,n,j,k$. Indeed, by \eqref{actionejks}, we get  
	\begin{align}
	\scal{{\Cch_s}(e^{s}_{jk}), {\Cch_s}(e^{r}_{mn}) }_{B}&=\varepsilon_{k-j}\varepsilon_{n-m} \gamma^{\omega}_{k,s}\gamma^{\omega}_{n,r}\int_{\overline{\Di}^c}\frac{1}{z^{k-j+1}\overline{z}^{n-m+1}}
	B(|z|^2)  d\lambda(z)  \nonumber 
	\\&=\pi\varepsilon_{k-j}\gamma^{\omega}_{k,s}\gamma^{\omega}_{n,r} \delta_{k-j,n-m}\int_{1}^{+\infty}\frac{1}{t^{k-j+1}} B(t)  dt \nonumber 
	\\&=\pi\varepsilon_{k-j}\gamma^{\omega}_{k,s}\gamma^{\omega}_{n,r}\delta_{k-j,n-m}\int_{0}^{1}u^{k-j-1}B(1/u)  du.  \label{scaler}
	\end{align}
	This proves the second assertion in $1)$, while the first assertion follows as  particular case by taking $m=j$.
	The result in $2)$ is exactly \eqref{scaler} with $m=j$ and $n=k$. 
\end{proof}

\begin{remark}
	In view of \eqref{scaler}, it is clear that the family $\left( {\Cch_s}^{\omega_\alpha}(e^{s}_{jk})\right)_{j,k,s}$ is not.
\end{remark}

Using Lemma \ref{lemaction}, we  give the explicit action of ${\Cch_s}^{\omega_\alpha}$ on the disc polynomials.  

\begin{proposition}\label{propactionZ}
	Let $\gamma >-1$. Then, for every nonnegative integers $m,n$, there exists some constant $c^{\gamma,\omega}_{m,n}$, depending in $\gamma$, $\omega$, $m$ and $n$, such that 
	\begin{align*}
	{\Cch_s}^{\omega_\alpha} (\mathcal{R}^{\gamma}_{m,n})(z) 
	= c^{\gamma,\omega}_{m,n}  \frac{z^m}{z^{n+1}}.
	\end{align*}
	For the weight function $\omega(t)=\omega_\alpha(t)=(1-t)^\alpha$, it is given explicitly by  
$$
c^{\gamma,\omega_\alpha}_{m,n} = \varepsilon_{n-m} 
\frac{(\gamma-\alpha)_m n!}{(\alpha+1)_{n+1}(\gamma+1)_m} .  
$$  
\end{proposition}

\begin{proof}
	By setting  
	$$c^{\gamma,j}_{m,n}
	=\frac{(-1)^{j} m!n!   }{(\gamma+1)_j j!(m-j)!(n-j)!},  
	$$
	we can rewrite the disc polynomials  \eqref{expZ} by mean of the  generic elements $e^{j}_{n-j,m-j}$ as
	\begin{align*}
	\mathcal{R}_{m,n}^{\gamma}(z,\bar{z})
	= \sum_{j=0}^{\minmn }c^{\gamma,j}_{m,n}e^{j}_{m-j,n-j}(z,\bar{z}).
	\end{align*}
	Subsequently, the linearity of the weighted Cauchy transform and Lemma \ref{lemaction} show that 
	\begin{align*}
	{\Cch_s}^{\omega_\alpha}(\mathcal{R}^{\gamma}_{m,n})(\xi) 
	&=\varepsilon_{n-m}\left( \sum_{j=0}^{m }c^{\gamma,j}_{m,n}\gamma^{\omega}_{n-j,j}\right) \frac{1}{z^{n-m+1}} = c^{\gamma,\omega}_{m,n}  \frac{z^m}{z^{n+1}},
	\end{align*}
	where the quantity  $\gamma^{\omega_\alpha}_{n-j,j}$, for  $\omega(t):=
	\omega_\alpha(t)=(1-t)^\alpha$, reduces to a beta function,
		\begin{align*}
	\gamma^{\omega_\alpha}_{n-j,j}=
	\int_0^1 t^{n-j} (1-t)^{\alpha+j} dt &= \frac{(n-j)!(\alpha+1)_j}{(\alpha+1)_{n+1}}. 
	\end{align*}
	 Therefore, the computation of the involved finite sum, that we denote $S^{\gamma,\alpha}_{m,n} $,
	 it turns out to be the value of a Gauss hypergeometric function at $1$. Indeed, we have 
	  \begin{align} 
S^{\gamma,\alpha}_{m,n} 
& = \frac{ n!}{(\alpha+1)_{n+1} } \sum_{j=0}^{m}\frac{(-m)^{j} (\alpha+1)_j}{ (\gamma+1)_j} \frac{1 }{ j!}  \nonumber
\\& = \frac{ n!}{(\alpha+1)_{n+1}} {_2F_1}\left( \begin{array}{c} -m,\alpha+1 \\ \gamma+1 \end{array}\bigg | 1\right) 
.  	\label{Smn}\end{align}
By mean of  Chu--Vandermonde identity 
$$ {_2F_1}\left( \begin{array}{c} -m,b \\ c \end{array}\bigg | 1\right) =\frac{ (c-b)_m }{(c)_m} ,$$
the explicit expression of the constant $c^{\gamma,\omega}_{m,n}$, for  $\omega(t)=\omega_\alpha(t)$, reduces to 
$$
c^{\gamma,\omega_\alpha}_{m,n} = \varepsilon_{n-m} 
\frac{(\gamma-\alpha)_m n!}{(\alpha+1)_{n+1}(\gamma+1)_m}.   
$$ 
\end{proof}

Subsequently,  it is clear that the following assertions hold trues
\begin{itemize}
	\item[1)] The range of $\Cch_s^{\omega_\alpha}$ restricted to $\AnTWDi_{n}(\Di)$ is finite dimensional vector space with dimension do not exceed  $n+1$, since $\mathcal{R}^{\gamma}_{m,n}(\xi,\overline{\xi})$ form an orthogonal basis of the true weighted Bergman spaces \cite{ElHartiElkGh2020A}.
	\item[2)] $\mathcal{R}_{m,n}^{\gamma}\in \ker({\Cch_s}^{\omega_\alpha})$, for any $m>n.$
	\item[3)] We have    ${\Cch_s}^{\omega_\alpha}(\mathcal{R}^{\gamma}_{m,n})=\mbox{const.}{\Cch_s}(e^{s}_{n,m}),$
	and for fixed $n$ and varying $m=0,1,2,\cdots $, they form an orthogonal system of is holomorphic functions in $L^{2,B}(\overline{\Di}^c)$. 
	\item[4)] 
	For $\alpha=\gamma$, the involved constant is exactly zero for any $m>0$, while for $m=0$, it reduces to 
	\begin{align*}
	{\Cch_s}^{\omega_\alpha}(\mathcal{R}^{\gamma}_{0,n})(\xi) 
	&=\frac{n!  }{(\gamma+1)_{n+1}}  \frac{1}{z^{n+1}}
	.
	\end{align*}
\end{itemize}



With the material presented in this section we are now able to prove our main result. 

\section{Proof of main result (Theorem \ref{thm1})} 

Notice first that since we are placed in the case of $\omega =\omega_\alpha$, we have to assume that $\alpha >(\gamma-1)/2 >-1$ (by Remark \ref{remCond}) to ensuring the boundedness of $\Cch_s^{\omega_\alpha}$,  the finiteness of the weight function $\omega_\gamma$ and therefore the fact that the disc polynomials is an orthogonal basis of $L^{2,\gamma}(\Di)$. According to the explicit expression of the constant $	c^{\gamma,\omega_\alpha}_{m,n}$ in Proposition \ref{propactionZ}, it is clear that $\Cch_s^{\omega_\alpha} \mathcal{R}^{\gamma}_{m,n} \neq 0$ if and only if $n <m$ and $(\gamma-\alpha)_m\neq 0$ with $m>\alpha-\gamma$. In particular,  $Span\{ \mathcal{R}^{\gamma}_{m,n}; m <n\} \subset \ker \Cch_s^{\omega_\alpha}$. 
For the determination of $\dim (\Cch_s^{\omega_\alpha}(\AnTWDi_{n}(\Di))) \leq n+1$,  two cases are to be distinguished $\gamma -\alpha\in \Z^-_0=\{0,1,2,\cdots\}$ and  $\gamma -\alpha\notin \Z^-_0$. 

Thus, if  $\gamma -\alpha\notin \Z^-_0$, then $(\gamma -\alpha)_m$ is not zero for every nonnegative integer $m\geq 0$, and 
$\Cch_s^{\omega_\alpha} \mathcal{R}^{\gamma}_{m,n} =0$ if and only if $m > n$.  
Subsequently, the restriction of the solid weighted Cauchy transform ${\Cch_s}$ to $\AnTWDi_{n}(\Di)$ is spanned by 
$$ \frac{z^m}{z^n+1}; \quad m=0, 1, 2, \cdots , n .$$ 
So that the dimension of is infected by the weight functions and is equal to $n+1$.
  
Now, when  $\gamma -\alpha \in \Z^-_0$, we have 
	\begin{align*}
{\Cch_s}^{\omega_\alpha} (\mathcal{R}^{\gamma}_{m,n})(\xi) 
= \left\{ \begin{array}{ll}
\displaystyle  0; & \mbox{ if } m> n \\ 
\displaystyle  \frac{(\gamma-\alpha)_m n!}{(\alpha+1)_{n+1}(\gamma+1)_m} \frac{z^m}{z^{n+1}} ; & \mbox{ if }    m \leq \min(n,\alpha -\gamma+1) \\
\displaystyle   0; & \mbox{ if } \alpha -\gamma+1  \leq m\leq n .   
\end{array}\right.
\end{align*}  
In this case, ${\Cch_s}^{\omega_\alpha} \ne 0$ if and only if  $m \leq \min(n,\alpha -\gamma+1)$. Hence, it follows
	$${\Cch_s}^{\omega_\alpha}(\AnTWDi_{n}(\Di)) = Span\left\{  \frac{z^m}{z^{n+1}} ; 0\leq m \leq  \min (n,\alpha-\gamma+1) \right\}$$
	and 
	$$\ker {\Cch_s^{\omega_\alpha}}|_{\AnTWDi_{n}(\Di)} =  
Span\left\{  \mathcal{R}^{\gamma}_{m,n} ;   m \geq  \min (n,\alpha-\gamma+1) \right\}.$$
Clearly the dimension of ${\Cch_s}^{\omega_\alpha}(\AnTWDi_{n}(\Di))$ is finite and given by 
$$N= \min (n,\alpha-\gamma+1) +1.$$

Moreover,  ${\Cch_s}^{\omega_\alpha}(\AnTWDi_{n}(\Di))_{n}$ is an increasing sequence of spaces.

\begin{remark}
	For $\gamma=\alpha$, the ranges ${\Cch_s}^{\omega_\alpha}(\AnTWDi_{n}(\Di))_{n}$ are of all of dimension one. 
\end{remark}
\begin{remark}
	The case of $\gamma > \alpha >(\gamma-1)/2 >-1$ is clearly contained in the case of $\gamma -\alpha \notin \Z^-_0$. Wile the case of $-1<\gamma < \alpha$ depends on the quantization of $\gamma - \alpha$.  
\end{remark}

%


\begin{thebibliography}{45}
	
%
\bibitem{Aharmim2015} Aharmim B., El Hamyani A., El Wassouli F., Ghanmi A.,
        {Generalized Zernike polynomials: operational formulae and generating functions}.
         Integral Transforms Spec. Funct. 26 (2015), no. 6, 395--410.
         

	\bibitem{Balk1991} Balk M.B., Polyanalytic functions. Mathematical Research, 63. Akademie-Verlag, Berlin, 1991.
	
\bibitem{Bell2016} Bell S.R., The Cauchy transform, potential theory and conformal mapping. Second edition. Chapman $\&$ Hall/CRC, Boca Raton, FL, 2016.


\bibitem{Calderon1977} A. P. Calderon, 
Cauchy integrals on Lipschitz curves and related operators. Proc. Nat. Acad. Sci. U.S.A. 74 (1977), no. 4, 1324?1327. 

 


	
\bibitem{Dunkl84}  Dunkl C.F.,
             { The Poisson kernel for Heisenberg polynomials on the disk}.
              Math. Z. 1984;187(4):527--547.


\bibitem{ElHartiElkGh2020A}   El Harti R.,  ElKachkouri A.,, A. Ghanmi A.,  A note on weighted poly-Bergman spaces. arXiv:2008.12764  math.CV 

 

\bibitem{ElHGhIn}   El Hamyani A., Ghanmi A., Intissar A.,  
Generalized Zernike polynomials: Integral representation and Cauchy transform. arXiv:1605.00281  math.CA



       
  
  
%
%
%
  	
 
  	

\bibitem{IsaevYoulmukhametov2004} 
Isaev K. P., Yulmukhametov R. S.,
The Laplace transform of functionals on Bergman spaces. (Russian. Russian summary)
Izv. Ross. Akad. Nauk Ser. Mat. 68 (2004), no. 1, 5?42; translation in
Izv. Math. 68 (2004), no. 1, 3?41


	
        


%
%
%

\bibitem{Koornwinder75}  Koornwinder T.H.,
           { Two-variable analogues of the classical orthogonal polynomials}.
           Theory and application of special functions, R.A. Askey (ed.), Academic Press, New York, 1975:435--495.
\bibitem{Koornwinder78}  Koornwinder T.H.,
        { Positivity proofs for linearization and connection coefficients of orthogonal polynomials satisfying an addition formula.}
        J Lond. Math. Soc. 1978;18(2):101--114.  

\bibitem{Lutsenko Yulmukhametov1991} V. I. Lutsenko, R. S. Yulmukhametov, Generalization of the Wiener-Paley
theorem to functionals in Smirnov spaces, (Russian), Trudy Mat. Inst. Steklov,
Vol. 200, 1991, 245?254.




\bibitem{Merenkov2000} 
Merenkov S.A., On the Cauchy transform of the Bergman space. Mat. Fiz. Anal. Geom. 7 (2000), no. 1, 119--127. 

\bibitem{Merenkov2013} 
Merenkov S.A., On the Cauchy transform of weighted Bergman space. Arxiv 2013.



\bibitem{NapalkovYulmukhametov1994}  Napalkov V.V., Jr. Yulmukhametov R.S.,  On the Cauchy transform of functionals on the Bergman space (Russian); translated from Mat. Sb. 185 (1994), no. 7, 77--86 Russian Acad. Sci. Sb. Math. 82 (1995), no. 2,

\bibitem{NapalkovYulmukhametov1997} Napalkov, V. V., Jr.; Youlmukhametov, R. S.
Criterion of surjectivity of the Cauchy transform operator on a Bergman space. Entire functions in modern analysis (Tel-Aviv, 1997), 261?267,
Israel Math. Conf. Proc., 15, Bar-Ilan Univ., Ramat Gan, 2001.
 327--336   


\bibitem{Ramazanov1999}  Ramazanov A. K., Representation of the Space of Polyanalytic Functions as a Direct Sum
of Orthogonal Subspaces. Application to Rational Approximations; Mathematical Notes, Vol. 66, No. 5, 1999.

\bibitem{Ramazanov2002}  Ramazanov A. K.,  
On the Structure of Spaces of Polyanalytic Functions. Mathematical Notes, vol. 72, no. 5, 2002, pp. 692?704.

\bibitem{RozenblumVasilevski2019} G. Rozenblum, N. Vasilevski,Toeplitz operators in polyanalytic Bergman type spaces, Func-tional Analysis and Geometry. Selim Grigorievich Krein Centennial. Contemporary Math-ematics,733, AMS, (2019).





	\bibitem{Vasilevski2000}{ Vasilevski N.L.},
{\it Poly-Fock spaces}.
Oper. Theory, Adv. App. 117 (2000)  371--386.


%
%
\end{thebibliography}
\end{document}